\documentclass[12pt, a4paper]{amsart}

\usepackage[hmargin=30mm, vmargin=25mm, includefoot, twoside]{geometry}
\usepackage[bookmarksopen=true]{hyperref}
\usepackage[usenames,dvipsnames]{xcolor}

\usepackage{amscd}
\usepackage{amsfonts,amssymb,verbatim}
\usepackage{latexsym}
\usepackage{mathrsfs}
\usepackage{stmaryrd}
\usepackage{xspace}
\usepackage{enumerate, paralist}
\usepackage{graphicx}
\usepackage{subcaption}
\usepackage[all]{xy}
\usepackage{extarrows}
\usepackage{tikz-cd}

\usepackage{txfonts, pxfonts}

\usepackage{amsthm}
\usepackage{amsmath}

\usepackage{todonotes}
\newtheorem{thm}{Theorem}[section]
 \newtheorem{cor}[thm]{Corollary}
 \newtheorem{lem}[thm]{Lemma}
 \newtheorem{prop}[thm]{Proposition}

\numberwithin{equation}{section}

 \theoremstyle{definition}
  \newtheorem{defn}[thm]{Definition}

 \theoremstyle{remark}
 \newtheorem{rem}[thm]{Remark}

\newtheorem*{claim*}{Claim}

\def\Dax{{\rm Dax}}

\def\ZZ{\mathbb{Z}}

\def\S{\mathscr{S}}

\def\B{\mathfrak{B}}

\begin{document}

\title[]{One internal stabilization is enough for resolving self-referential tubes}

\author{Qizheng You}

\address[Q. You]{School of Mathematical Sciences, Peking University, China.}
\email{qizhengyou@stu.pku.edu.cn}

\date{}

\thanks{}

\begin{abstract}
  This paper shows that one internal stabilization is enough to resolve all the self-referential tubes on an embedding.
  As an application, this paper gives a complete isotopy classification 
  for embedded surfaces with a common geometric dual in orientable $S^2$-bundle over a surface,
  provided that the embedded surfaces have greater genus than the base surface of the bundle. 
\end{abstract}

\date{\today}
\maketitle

\parskip 4pt

%\textit{Mathematics Subject Classification} (2020): 20F65, 20F69, 51F30, 43A07.


\section{Introduction}\label{sec:intro}

Internal stabilization is an important operation for eliminating obstructions to promoting homotopy to isotopy, as well as topological isotopy to smooth isotopy. 
As recalled in Definition \ref{def:mkpt}, an \emph{internal stabilization} of an embedded surface is obtained by attaching an extra $1$-handle. 
When the guiding arc of the $1$-handle represent the trivial element in the fundamental group, we call such procedure a \emph{standard internal stabilization}.
A fundamental result of Baykur and Sunukjian \cite{BS16} shows that homologous embedded surfaces become smoothly isotopic after sufficiently many $1$-handle stabilizations. 
Moreover, under the assumption that the inclusion-induced homomorphism
$\pi_1(\partial \nu \Sigma_i)\rightarrow \pi_1(X\setminus \nu \Sigma_i)$
is surjective for $i=0,1$, these stabilizations may be chosen to be standard.

A natural question is how many stabilizations are actually necessary, and in particular, whether a single stabilization might suffice. 
Early evidence in this direction was provided by Baykur and Sunukjian in \cite{BS16}, 
who showed that many previously known constructions of exotically knotted surfaces
(including those arising from rim surgery, twist-rim surgery, annulus surgery, and tangle surgery)
become smoothly isotopic after a single internal stabilization. 
This motivated the ``one stabilization is enough'' phenomenon and the question of how broadly it persists.
There is also substantial evidence for a one-stabilization phenomenon in the setting of external stabilization. 
Auckly, Kim, Melvin, Ruberman, and Schwartz proved in \cite{AKMRS19} that two homologous ordinary closed oriented surfaces of the same genus 
with simply connected complements become smoothly isotopic after a single external stabilization (i.e.\ taking connected sum of the external $4$-manifold with $S^2 \times S^2$).
On the other hand, the one-stabilization phenomenon does not hold in full generality. 
In \cite{Gu22}, Guth constructed exotic disks in $B^4$ with arbitrarily large internal stabilization distance. 
Thus, in the presence of boundary, there is no uniform bound on the number of internal stabilizations required. 
Further counterexamples have been found for closed surfaces. 
For example, Auckly in \cite{A26} showed that the internal stabilization distance 
of closed surfaces in some simply-connected $4$-manifolds can be arbitrarily large.

A related source of obstructions to isotopy arises from self-referential tubes. 
Gabai introduced in \cite{G21} the notions of self-referential tubes. 
Moreover, according to \cite{KT24}, for embedded disks with a common geometric dual in smooth $4$-manifolds without $2$-torsion in fundamental groups, 
self-referential tubes are the only obstructions from homotopy to smooth isotopy.
In this paper, we show that a single internal stabilization suffices to eliminate the obstruction arising from self-referential tubes. 
Our result provides further evidence for the one-stabilization phenomenon within this particular class of obstructions, 
especially for neatly embedded surfaces with boundary.

In this paper, we work in the smooth category,
and assume that all the surfaces and manifolds mentioned are oriented. 
Our main result is stated as follows.

\begin{thm}\label{thm:tech}
  Let $M$ be a $4$-manifold with boundary.
  Suppose $i:S \hookrightarrow M$ is a neat embedding of compact surface with a boundary geometric dual. 
  Denote the embedding obtained by adding self-referential tubes on $i$ as $j:S \hookrightarrow M$.
  After taking an internal stabilization on $i$ and $j$ respectively, 
  the resulting embeddings are isotopic relative to the boundary. 
\end{thm}

The construction of self-referential tubes is due to the work \cite{G21} of Gabai, and is reviewed in Subsection \ref{subsec:srt}.
We prove Theorem \ref{thm:tech} by constructing the isotopy via the light bulb lemma of Gabai in \cite{G20}. 
After combining Theorem \ref{thm:tech} with the result \cite[Proposition 1.7]{KT24} of Kosanović and Teichner, 
we directly obtain the following corollary.

\begin{cor}\label{cor:main}
  Let $M$ be a $4$-manifold with boundary, satisfying that $\pi_1(M)$ has no $2$-torsion.
  Suppose $i_1$ and $i_2$ are neat embeddings of disks that are homotopic relative to the boundary.
  We assume that $i_1$ and $i_2$ admit a common boundary geometric dual $G \subset \partial M$. 
  After taking an internal stabilization on $i_1$ and $i_2$ respectively, 
  the resulting embeddings are isotopic relative to the boundary. 
\end{cor}

Let $\Sigma$ be the closed surface with genus $g \geq 1$, and $X$ be an orientable $S^2$-bundle over $\Sigma$. 
Fix an $S^2$-fiber of $X$ and denote it by $G$.
For an embedded surface $\Sigma' \subset X$, 
we say $G$ is a \emph{geometric dual} of $\Sigma'$,
if $G$ and $\Sigma'$ intersect transversely and positively at one point.
As an application of Theorem \ref{thm:tech}, we state the following theorem regarding the isotopy classification of the embedded high-genus surfaces in $X$.

\begin{thm}\label{thm:main}
  Let $\Sigma'_1, \Sigma'_2 \subset X$ be two embedded closed surfaces with genus-$g'$ in $X$.
  Here $g' > g$.
  Suppose the following conditions holds.
  \begin{itemize}
    \item[(1)] The sphere $G$ is a common geometric dual of $\Sigma'_1$ and $\Sigma'_2$.
    \item[(2)] After parameterizations, $\Sigma'_1, \Sigma'_2$ are homotopic.
  \end{itemize}
  Then $\Sigma'_1$ and $\Sigma'_2$ are ambiently isotopic.
\end{thm}

Let $\mathscr{E}$ be the set consisting of the embedded genus-$g'$ surfaces in $X$ with the geometric dual $G$.
The following is a corollary of Theorem \ref{thm:main}. 

\begin{cor}\label{cor:main}
  There is a bijection 
  $$ \Phi: \mathscr{E}/\textit{ambient isotopy} \longrightarrow \ZZ. $$
\end{cor}

We briefly sketch the construction of the isotopy in Theorem \ref{thm:main}.
An initial isotopy can be chosen to make the two surfaces identical in a tubular neighbourhood of $G$. 
Then we apply Lemma \ref{lem:first isotopy} to isotope the surface to a genus-$g$ embedded surface after $g'-g$ standard internal stabilizations.
According to Proposition \ref{prop:2526}, after excluding the tubular neighbourhood of the genus-$g$ embedded surfaces in $X$,
we obtain two homotopic embeddings of surfaces with boundary.
Finally we construct the isotopy between these two embeddings via Theorem \ref{thm1}.
Corollary \ref{cor:main} can be proved directly by Theorem \ref{thm:main} and Proposition \ref{prop:hmtp}.

The isotopy problem of embedded surfaces with geometric dual was studied in the seminal work \cite{G20} of Gabai. 
Later in \cite{LWXZ25}, 
Lin, Wu, Xie and Zhang gave the complete isotopy classification of embedded genus-$g$ surfaces with a common geometric dual in $\Sigma \times S^2$. 
The similar problem for the nontrivial orientable $S^2$-bundle over $\Sigma$ was studied in \cite{Y26}.
This paper considers the isotopy problem for embedded genus-$g'$ surfaces in orientable $S^2$-bundle over $\Sigma$. 
The result for embedded high-genus surfaces is quite different from the genus-$g$ case.
For genus-$g$ surfaces, there are infinitely many surfaces, which are homotopic but pairwise non-isotopic, according to \cite{LWXZ25,Y26}.
However, for genus-$g'$ surfaces with common geometric dual, Theorem \ref{thm:main} shows that homotopy implies isotopy.

\begin{rem}
  We remark here that the procedure, utilizing Theorem \ref{thm:tech} to 
  discuss the isotopy classification problem of embedded high-genus surfaces, 
  can be applied to arbitrary $4$-manifolds. 
  However, a careful analysis is needed to prove an analogy of Proposition \ref{prop:2526}.
\end{rem}

The paper is organized as follows.
In Section \ref{sec:pre}, we review the notions of self-referential tubes, compressing disks and the light bulb lemma.
In Section \ref{sec:tech}, we introduce the notion of internal stabilizations of embeddings, and prove Theorem \ref{thm:tech}. 
In Section \ref{sec:main}, we are devoted to prove Theorem \ref{thm:main} and Corollary \ref{cor:main}.

\subsection*{Acknowledgement} 
The author would like to thank his advisor Yi Xie for his continuous guidance and assistance throughout the entire research process.
The author is partially supported by NSFC 12341105.

\subsection*{AI Declaration} 
AI was used only for text polishing, grammar checking, and guidance on creating figures using the TikZ package.

\section{Preliminaries}\label{sec:pre}

In this section, we are going to review several notions, including the self-referential tubes, compressing disks and the light bulb lemma.

We recall the definition of neat maps and embedding spaces.

\begin{defn} 
  A smooth map $i:X \rightarrow Y$ between two smooth manifolds with boundaries is \emph{neat}, 
  if $i^{-1}(\partial Y) = \partial X$, and $i$ is transverse to $\partial Y$.
\end{defn}

\begin{defn}\label{emb1}
  Let $X,Y$ be smooth manifolds with boundaries.
  \begin{itemize}
    \item[(1)] Define ${\rm Emb}(X,Y)$ to be the space consisting of all the neat embeddings $i$ of $X$ into $Y$, with $C^\infty$-topology.
    \item[(2)] Given $i_0 \in {\rm Emb}(X,Y)$, define ${\rm Emb}_\partial (X,Y)$ to be the subspace of ${\rm Emb}(X,Y)$ 
    consisting of the embeddings $i:X \rightarrow Y$ satisfying $i|_{\partial X} = i_0 |_{\partial X}$.
  \end{itemize}
\end{defn}

\subsection{The self-referential tubes}\label{subsec:srt}
We review the constructions of the self-referential tubes of Gabai in \cite{G21}.
Suppose that $M$ is a $4$-manifold with boundary.
Let $i:D^2 \hookrightarrow M$ be a neat embedding of $2$-disk in $M$.
We denote $I = [0,1]$.
Let $\gamma: I \rightarrow {\rm int}(M)$ be a path satisfying the following conditions as in Figure \ref{pic1}.
\begin{itemize}
  \item[(1)] It holds that $\gamma^{-1}(i(D^2)) = \{0\}$.
  \item[(2)] The only double point of $\gamma$ is $\gamma(t_0) = \gamma(1)$, where $0 < t_0 < 1$. 
\end{itemize}

\begin{figure}[htbp]
    \centering
    \begin{tikzpicture}[thick, scale=0.8, transform shape]
        \useasboundingbox (-1.5, -2.5) rectangle (8, 2.8);
        \draw (0,0) ellipse (1 and 2);
        \draw[white, line width=7pt] (1, -0.25) -- (1, 0.25);
        \draw (0,0) -- (4,0);
        \draw (4,0) 
            .. controls (4.8, 3.5) and (7.5, 2.5) .. (7.5, 1) 
            .. controls (7.5, -0.5) and (5, 0)    .. (4,0);
        \fill (0,0) circle (2.5pt);
        \fill (4,0) circle (2.5pt);
        \node[above] at (0, 1.1) {$i$};
        \node[above] at (6.2, 1.8) {$\gamma$};
        \node[below=4pt] at (4, 0) {$\gamma(t_0) = \gamma(1)$};
    \end{tikzpicture}
    \caption{The defining data for a self-referential tube} 
    \label{pic1}
\end{figure}
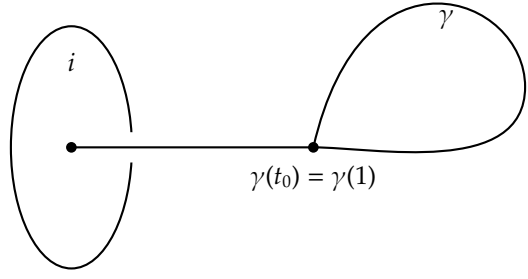

Denote the meridian sphere of $\gamma$ at $t=t_0$ as $m: S^2 \hookrightarrow M$. 
We may shrink the path $\gamma$ slightly at the endpoint $\gamma(1)$ to get a path $\tilde{\gamma}$, such that $\tilde{\gamma}(1) \in m(S^2)$.
After tubing $i$ to $m$ along $\tilde{\gamma}$, we obtain a neat embedding $i':D^2 \hookrightarrow M$.
Let $g$ be the element in $\pi_1(M,i(D^2))$ represented by the concatenation of $\gamma$ and $\overline{\gamma |_{[0,t_0]}}$.
Here the notation $\overline{\sigma}$ refers to the orientation-reversing path of a path $\sigma$. 

Since the homotopy of tubes in $4$-manifolds can be tamed to isotopy,
the isotopy class of $i'$ relative to the boundary only depends on the element $g$.
Thus, when considering the isotopy classification problem, 
we may denote $i'$ by $i_g$ for the sake of clarity, and call $i_g$ the embedding obtained by \emph{adding a self-referential tube} on $i$ along $g$. 
Figure \ref{pic2} is an illustration of the embedding $i_g$.
Note that $i$ and $i_g$ are homotopic relative to the boundary.
Regarding the relative Dax invariant, Gabai in \cite{G21} showed the equation
$$ \Dax (i,i') = [g] + [g^{-1}]. $$
For the definition of the Dax invariant, we refer the readers to \cite[Section 3]{G21}.

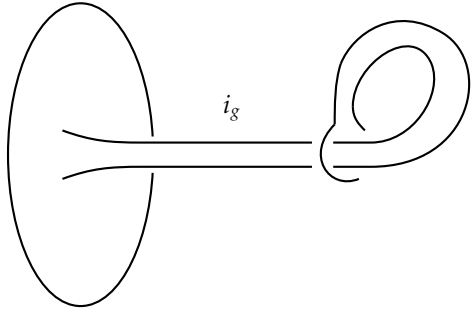
\begin{figure}[htbp]
    \centering
    \begin{tikzpicture}[thick, scale=0.8, transform shape]
        \draw (0,0) ellipse (1.2 and 2.5);
        \draw[white, line width=7pt] (1.2, -0.3) -- (1.2, 0.3);
        \draw (-0.3, 0.4) to[out=-20, in=180] (1, 0.2) -- (4.8, 0.2);
        \draw (-0.3, -0.4) to[out=20, in=180] (1, -0.2) -- (4.8, -0.2);
        \draw [white, line width=8pt] 
            (4.0, 1.6) -- (4.0, -0.5);
        \draw (4.8, -0.2) 
            .. controls (6.5, -0.2) and (6.8, 1.5) .. (6.0, 2.0) 
            .. controls (5.2, 2.5) and (4.5, 2.0) .. (4.3, 1.5);
        \draw (4.3, 1.5)
            .. controls (4.2, 1.2) and (4.2, 0.8) .. (4.2, 0.5);
        \draw (4.2, 0.5)
            .. controls (3.7, 0) and (4.1, -0.6) .. (4.6, -0.4);
        \draw (4.5, 0.8)
            .. controls (4.5, 1.5) and (5.5, 2.2) .. (5.8, 1.5)
            .. controls (6.0, 1.0) and (5.5, 0.2) .. (4.8, 0.2);
        \draw (4.5, 0.8)
            .. controls (4.5, 0.6) and (4.6, 0.5) .. (4.7, 0.4);
        \node[above] at (2.5, 0.4) {$i_g$};
    \end{tikzpicture}
    \caption{The definition of a self-referential tube} 
    \label{pic2}
\end{figure}

\subsection{Compressing disks}
We review the notions of $G$-inessential embeddings and compressing disks in \cite{G20}.
Suppose $M$ is a $4$-manifold (may with boundary), and $\Sigma$ is a compact surface.

\begin{defn}
  Let $i: \Sigma \hookrightarrow M$ be a neat embedding.
  An embedding $G: S^2 \hookrightarrow M$ is called a \emph{geometric dual} of $i$, if $i$ and $G$ intersect positively and transversely at only one point $p \in \Sigma$.
  Furthermore, the embedding $i:\Sigma \hookrightarrow M$ is called \emph{$G$-inessential}, 
  if the induced map $i_*: \pi_1(\Sigma \setminus p) \rightarrow \pi_1(M \setminus G)$ is trivial.
\end{defn}

\begin{defn}\label{def:cpd}
  For a neat embedding $i:\Sigma \hookrightarrow M$, 
  a \emph{compressing disk} for $i$ is an embedded $2$-disk $D \subset {\rm int}(M)$, 
  satisfying the following conditions.
\begin{itemize}
  \item[(1)] It holds that $D \cap i(\Sigma) = \partial D$.
  \item[(2)] The section of the normal bundle $\nu(\partial D \subset i(\Sigma))$ can be extended to a section of the normal bundle $\nu(D \subset M)$.
\end{itemize}
\end{defn}

Note that, given an embedding $i:\Sigma \hookrightarrow M$ with a compressing disk $D$,
we may exclude a tubular neighbourhood of $\partial D$ in $i(\Sigma)$, 
and glue two parallel copies of $D$ back to obtain a new embedding $i_D: \Sigma_- \hookrightarrow M$.
Here $\Sigma_-$ is the closed surface with genus $g(\Sigma_-) = g(\Sigma) - 1$.
We call $i_D$ the embedding obtained by compressing $i$ along $D$.
Analogously, given $n$ disjoint compressing disks $D_1, \cdots, D_n$ of $i$,
we denote the obtained embedding after compressing along these disks by $i_{D_1,\cdots,D_n}$.

Regarding the existence of compressing disks, we have the following result according to \cite[Lemma 9.2, Lemma 9.3]{G20}.

\begin{lem}\label{lem:gabai}
  Let $i:\Sigma \hookrightarrow M$ be an embedding, admitting a geometric dual $G:S^2 \hookrightarrow M$ at $p \in \Sigma$.
  Suppose that $\alpha_1,\cdots,\alpha_n$ are disjoint simple closed curves in $i({\rm int}(\Sigma))$ satisfying that
  \begin{itemize}
    \item[(1)] for any $k$, $\alpha_k$ is disjoint from the geometric dual $G$;
    \item[(2)] the complement $i(\Sigma) \setminus (\alpha_1 \cup \cdots \cup \alpha_n)$ is connected;
    \item[(3)] for any $k$, $\alpha_k$ is homotopically trivial in $M \setminus G(S^2)$.
  \end{itemize}
  Then there are disjoint compressing disks $D_1,\cdots,D_n$ for $i$, such that for any $k$, $\partial D_k = \alpha_k$.
\end{lem}

\subsection{The light bulb lemma}
When considering the isotopy of embedded surfaces with a geometric dual,
there is a useful lemma named the light bulb lemma introduced in \cite{G20}.
Let $M$ be a $4$-manifold with boundary, and let $\Sigma$ be a surface with boundary. 

\begin{lem}\label{lem:lbl}\cite[Lemma 2.3]{G20}
  Let $i: \Sigma \hookrightarrow M$ be a neat embedding with a boundary geometric dual $G: S^2 \hookrightarrow \partial M$. 
  We denote the intersection of $i(\Sigma)$ and $G(S^2)$ by $p$.
  We assume the following conditions hold as in Figure \ref{pic3}.
  \begin{itemize}
    \item[(1)] Suppose $\alpha_0$ and $\alpha_1$ are two embedded arcs in $M$ sharing the same endpoints. 
    The concatenation of the arcs $\alpha_0$ and $\overline{\alpha_1}$ bounds an embedding disk $E$ in $M$.
    The disk $E$ and $i(\Sigma)$ intersect transversely at a point $q$.
    \item[(2)] There is an isotopy $\alpha_t: I \rightarrow M$ defined by swiping the arc from $\alpha_0$ to $\alpha_1$.
    After applying the arc-pushing map, we obtain a smooth ambient isotopy $f_t: M \rightarrow M$ for $t \in [0,1]$.
    satisfying that $f_0 = {\rm id}_{M}$, and $f_1$ sends $\alpha_0$ to $\alpha_1$, and the ambient isotopy $f_t$ supports on a neighbourhood $N$ of $E$.
    \item[(3)] We identify a closed neighbourhood of $\alpha_0$ to $B^3 \times I$.
    Suppose that $i(\Sigma) \cap N = C \cup B$.
    Here $B$ lies in ${\rm int}(B^3) \times I$, and $C$ is an embedded disk in $i(\Sigma)$ containing $q$. 
    \item[(4)] We suppose that $p$ and $q$ lie in the same component of $i(\Sigma) \setminus B$.
  \end{itemize}

  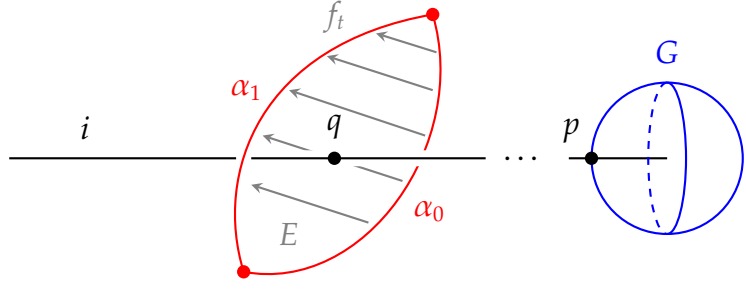
\begin{figure}[htbp]
    \centering
    \begin{tikzpicture}[thick, >=stealth]
    \begin{scope}[gray, ->]
        \draw (0.25, -0.85) -- (-1.3, -0.35);
        \draw (0.7, -0.3) -- (-1.15, 0.3); 
        \draw (1.0, 0.3) -- (-0.8, 0.9);
        \draw (1.1, 0.9) -- (-0.3, 1.35);
        \draw (1.15, 1.4) -- (0.35, 1.65);
    \end{scope}
    \draw (-4.5, 0) -- (-1.5, 0);
    \node at (2.3, -0.05) {$\cdots$};
    \draw (2.9, 0) -- (4.2, 0);
    \def\B{-1.4, -1.5}
    \def\T{1.1, 1.9}
    \draw[red] (\B) .. controls (-1.8, -0.2) and (-1.2, 1.5) .. (\T);
    \draw[red] (\B) .. controls (0.2, -1.8) and (1.6, 0.2) .. (\T);
    \draw [white, line width=6pt] 
    (-0.8, 0) -- (1.5, 0);
    \draw (-1.3, 0) -- (1.8, 0);
    \def\Cx{4.2}
    \def\Cy{0}
    \def\R{1.0}
    \draw[blue] (\Cx, \Cy) circle (\R);
    \draw[blue] (\Cx, \Cy-\R) arc (-90:90:0.25cm and \R cm);
    \draw[blue, dashed] (\Cx, \Cy+\R) arc (90:270:0.25cm and \R cm);
    \fill[red] (\B) circle (2.5pt); 
    \fill[red] (\T) circle (2.5pt); 
    \fill (-0.2, 0) circle (2.5pt); 
    \fill (3.2, 0) circle (2.5pt); 
    \node[above] at (-3.5, 0.1) {$i$};
    \node[left, red] at (-1.0, 0.9) {$\alpha_1$};
    \node[right, red] at (0.7, -0.7) {$\alpha_0$};
    \node[above, gray] at (-0.2, 1.5) {$f_t$};
    \node[below, gray] at (-0.8, -0.7) {$E$};
    \node[above] at (-0.2, 0.1) {$q$};
    \node[above left] at (3.2, 0.05) {$p$};
    \node[above, blue] at (\Cx, \R + 0.1) {$G$};
    \end{tikzpicture}
    \caption{The conditions of the light bulb lemma}
    \label{pic3}
\end{figure}

  Then there is an ambient isotopy $g_t: M \rightarrow M$ for $t \in [0,1]$, satisfying the following.
  \begin{itemize}
    \item[(a)] It holds that $g_0 = {\rm id}_M$.
    \item[(b)] The restriction of $g_1$ to $i(\Sigma) \setminus B$ is the identify map, and the restriction of $g_1$ to $B$ is $f_1|_{B}$.
    \item[(c)] For any $t \in [0,1]$, the diffeomorphism $g_t$ fixes a neighbourhood of $G$ and $i(\Sigma) \setminus B$ pointwise.
  \end{itemize}
\end{lem}

\section{Resolving Self-referential Tubes via an Internal Stabilization}\label{sec:tech}

In this section, we will state and prove the main theorem of this paper. 
This theorem shows that, by utilizing an internal stabilization, 
we can resolve any self-referential tubes via isotopy.
Let $M$ be a $4$-manifold with boundary, and let $\Sigma$ be an oriented surface with boundary. 

\subsection{Embeddings after taking internal stabilizations}

In this subsection, we review the notion of the embedding after taking internal stabilizations.

\begin{defn}
  Suppose $i: \Sigma \hookrightarrow M$ is a neat embedding.
  Given an arc embedding $\gamma: I \hookrightarrow M$ with endpoints in $i({\rm int}(\Sigma))$. 
  We may add a tube along $\gamma$ to obtain an embedding $i^\gamma: \Sigma_+ \hookrightarrow M$, respecting the orientation of $\Sigma$.
  Here $\Sigma_+$ is the surface with genus $g(\Sigma_+) = g(\Sigma) + 1$.
  We refer to this procedure as \emph{taking an internal stabilization} along the \emph{guiding arc} $\gamma$ at $p$ on the embedding $i$.
\end{defn}

\begin{rem}\label{rem:intstb}
  Fixing the guiding arc $\gamma$, there are two distinct ways to tube along $\gamma$, corresponding to the elements of $\pi_1(SO(3)) \cong \ZZ/2$.
  We remark here that, the images of the two resulting embeddings are ambiently isotopic.
  After identifying these images via an isotopy, the two parameterizations differ via a Dehn twist along a meridian of the tube.
\end{rem}

\begin{rem}\label{rem:intstb2}
  The isotopy class (relative to the boundary) of the embedding $i^\gamma$ is invariant under the isotopy of $\gamma$.
  Here, we allow the endpoints of $\gamma$ to move during the isotopy, provided that they remain on $i({\rm int}(\Sigma))$ throughout. 
\end{rem}

\begin{defn}\label{def:mkpt}
  Suppose $i: \Sigma \hookrightarrow M$ is a neat embedding, and let $p \in \operatorname{int}(\Sigma)$. 
  As shown in Figure \ref{pic4}, we can attach a tube to $i(\Sigma)$ 
  such that the tube is strictly contained within a small $4$-ball neighborhood of $i(p)$ in $M$, 
  and its intersection with $i(\Sigma)$ lies entirely within a small $2$-disk neighborhood of $i(p)$.
  We refer to this procedure as \emph{taking a standard internal stabilization} at $p$ on the embedding $i$, 
  and denote the resulting embedding by $i^p: \Sigma_+ \hookrightarrow M$.
  Furthermore, given $n$ distinct points $p_1, \dots, p_n \in \operatorname{int}(\Sigma)$,
  we can define the embedding $i^{p_1,\dots,p_n}$ named \emph{taking $n$ standard internal stabilizations on $i$} 
  in an analogous manner by locally attaching $n$ tubes at the respective points.
\end{defn}

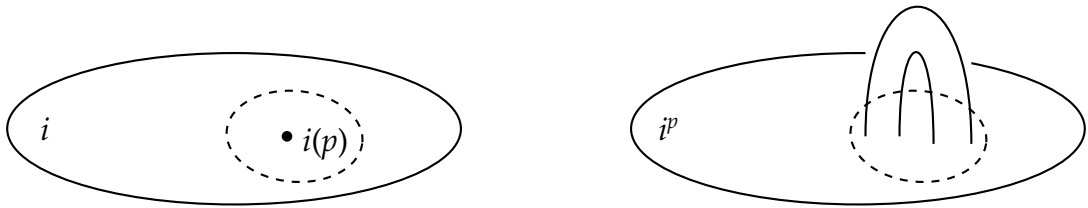
\begin{figure}[htbp]
    \centering
    \begin{minipage}{0.45\textwidth}
        \centering
        \begin{tikzpicture}[thick]
            \useasboundingbox (-3.2, -1.2) rectangle (3.2, 1.8);
            \draw (0,0) ellipse (3 and 1);
            \node at (-2.5, 0) {$i$};
            \draw[dashed, rotate around={-5:(0.8, -0.1)}] (0.8, -0.1) ellipse (0.9 and 0.6);
            \fill (0.7, -0.1) circle (2pt);
            \node[right] at (0.75, -0.2) {$i(p)$};
        \end{tikzpicture}
    \end{minipage}\hfill
    \begin{minipage}{0.45\textwidth}
        \centering
        \begin{tikzpicture}[thick]
            \useasboundingbox (-3.2, -1.2) rectangle (3.2, 1.8);
            \draw (0,0) ellipse (3 and 1);
            \draw [white, line width=10pt] (0.1, 0.9) -- (1.5, 0.9);
            \node at (-2.5, 0) {$i^p$};
            \draw[dashed, rotate around={-5:(0.8, -0.1)}] (0.8, -0.1) ellipse (0.9 and 0.6);
            \draw (0.1, -0.1) .. controls (0.1, 2.2) and (1.5, 2.2) .. (1.5, -0.2);
            \draw (0.55, -0.1) .. controls (0.55, 1.4) and (1.0, 1.4) .. (1.0, -0.2);
        \end{tikzpicture}
    \end{minipage}
    \caption{Taking a standard internal stabilization at $p$ on $i$.}
    \label{pic4}
\end{figure}

\begin{rem}\label{rem:mkpt}
  We remark that for different choices of the points $p_1, \dots, p_n$ in $\operatorname{int}(\Sigma)$,
  the isotopy classes (relative to the boundary) of the embeddings $i^{p_1,\dots,p_n}$ are identical. 
  This is because we can always isotope the these points to a standard configuration in $\operatorname{int}(\Sigma)$.
  Consequently, when considering only the isotopy class, we denote the embedding after $n$ internal stabilizations on $i$ simply by $i^{\bullet n}$.
\end{rem}

\begin{rem}\label{rem:mkpt2}
  Note that taking an internal stabilization is a local operation.
  Therefore, given isotopic embeddings $i_1$ and $i_2$, for positive integer $n$,
  the embeddings $i_1^{\bullet n}$ and $i_2^{\bullet n}$ are still isotopic. 
\end{rem}

As an application of the light bulb lemma, 
Gabai in \cite{G20} showed that an embedding with a compressing disk is isotopic to an internal stabilization on the embedding after compressing.

\begin{prop}\label{prop2.9}
  Suppose $i: \Sigma \hookrightarrow M$ is a neat embedding with a boundary geometric dual.
  Let $D_1, \cdots, D_n$ be $n$ disjoint compressing disks of $i$.
  Then the embedding $i$ is isotopic to the embedding $(i_{D_1, \cdots, D_n})^{\bullet n}$ relative to the boundary.
\end{prop}

The proof of Proposition \ref{prop2.9} is essentially given in the proof of \cite[Theorem 9.7]{G20}.
For the sake of clarity, we restate its proof here.

\begin{proof}
  The original embedding $i$ can be recovered from the compressed surface $i_{D_1, \cdots, D_n}$ by attaching $n$ tubes.
  Each of these tubes, which can be identified with $S^1 \times I$, extends to a solid tube $D^2 \times I$ in $M$.
  These solid tubes intersect the compressed surface exactly at their bases $D^2 \times 0$ and $D^2 \times 1$.
  Furthermore, by construction, these $n$ tubes are pairwise disjoint.

  After applying an ambient isotopy, we may assume that there are $n$ small, pairwise disjoint $4$-balls, each of which intersects the compressed surface in a single standard disk that contains the bases of a single solid tube.
  By the light bulb lemma, these solid tubes can be isotoped to be 3-dimensional neighborhoods of tiny standard arcs in these disjoint $4$-balls with endpoints on the compressed surface.
  That is because the embedding $i$ has a boundary geometric dual and the complement of the tubes in the compressed surface is connected.

  After this isotopy, each attached tube is strictly contained within a small $4$-ball neighborhood of a point, and its intersection with the compressed surface lies entirely within a small $2$-disk neighborhood. 
  By Definition \ref{def:mkpt}, this local topological configuration is exactly the procedure of taking internal stabilization on the embedding. 
  Since this procedure applies to each of the $n$ disjoint tubes locally, the embedding $i$ is isotopic to $(i_{D_1, \cdots, D_n})^{\bullet n}$ relative to the boundary.
\end{proof}

\subsection{Resolving self-referential tubes}

Now we are ready to restate Theorem \ref{thm:tech}, and prove the following result.

\begin{thm}\label{thm1}
  Suppose $i: \Sigma \hookrightarrow M$ is a neat embedding with a boundary geometric dual.
  Let $g$ be an element of $\pi_1(M)$, and $\gamma: I \hookrightarrow M$ be an arc embedding with endpoints on $i({\rm int}(\Sigma))$.
  Then the two embeddings $i_g^{\gamma}$ and $i^{\gamma}$ are isotopic relative to the boundary.
\end{thm}

\begin{proof}
  We prove the theorem by constructing an isotopy from $i_g^{\gamma}$ to $i^{\gamma}$.
  According to Remark \ref{rem:intstb2}, we may assume that the endpoints of $\gamma$, where we take the internal stabilization, lies in a $2$-disk neighbourhood $D_0$ centered at the base of the self-referential tube in $i(\Sigma)$. 

  \begin{figure}[htbp]
    \centering
    \begin{subfigure}[b]{0.45\textwidth}
        \centering
        \resizebox{\linewidth}{!}{
        \begin{tikzpicture}[thick]

          \draw (0,0) ellipse (1.2 and 2.5);
          \node[above] at (-0.2, 1.5) {$D_0$};
          
          \draw[white, line width=15pt] (1.0, 0.8) -- (1.5, 0.8);
          \draw[white, line width=12pt] (1.0, -0.5) -- (1.5, -0.5);

          \fill[red] (0, 1.0) circle (2pt);
          \fill[red] (0, 0.6) circle (2pt);
          \draw[red] (0, 1.0) -- (1.5, 1.0);
          \draw[red] (0, 0.6) -- (1.5, 0.6);
          \draw[red] (1.5,0.6) 
            .. controls (4, -1) and (4, 2.6) .. (1.5, 1.0);

          \fill[blue] (0, -0.5) circle (2pt);
          \draw[blue] (0,-0.5) -- (4,-0.5);
          \draw[blue] (4,0) 
            .. controls (4.8, 3.5) and (7.5, 2.5) .. (7.5, 1) 
            .. controls (7.5, -1) and (5, -0.5)    .. (4,-0.5);

          \draw[white, line width=8pt] (3.5, -1) -- (3.5, 0);
          \draw (4.5, -0.38) arc (15:345:0.5);
          
          \node at (2, 1.5) {$\gamma$};
          \node at (4,-1.3) {$m$};
          \node at (2, -1) {$i^{\gamma}_g$};
        \end{tikzpicture}
        }
        \caption{}
        \label{fig:sub1}
    \end{subfigure}
    \hfill
    \begin{subfigure}[b]{0.45\textwidth}
        \centering
        \resizebox{\linewidth}{!}{
        \begin{tikzpicture}[thick]

          \draw (0,0) ellipse (1.2 and 2.5);
          \node[above] at (-0.2, 1.5) {$D_0$};
          
          \draw[white, line width=15pt] (1.0, 0.8) -- (1.5, 0.8);
          \draw[white, line width=15pt] (1.0, -0.4) -- (1.5, -0.4);

          \fill[red] (0, 1.0) circle (2pt);
          \draw[red] (0, 1.0) -- (1.5, 1.0);
          \draw[red] (0.5, 0.6) -- (1.5, 0.6);
          \draw[red] (1.5,0.6) 
            .. controls (4, -1) and (4, 2.6) .. (1.5, 1.0);
          \draw[red] (0.5, 0.6)
            .. controls (-0.2, 0.6) and (-0.2, -0.3) .. (0.5,-0.3);
          \draw[red] (0.5,-0.3) -- (4,-0.3);
          \draw[red] (4.2,-0.1) 
            .. controls (4.8, 3.2) and (7.2, 2.2) .. (7.2, 1) 
            .. controls (7.2, -1) and (5, -0.3)    .. (4,-0.3);

          \fill[blue] (0, -0.5) circle (2pt);
          \draw[blue] (0,-0.5) -- (4,-0.5);
          \draw[blue] (4,0) 
            .. controls (4.8, 3.5) and (7.5, 2.5) .. (7.5, 1) 
            .. controls (7.5, -1) and (5, -0.5)    .. (4,-0.5);

          \draw[white, line width=10pt] (3.5, -1) -- (3.5, 0);
          \draw (4.3, -0.17) arc (45:345:0.5);

          \fill[red] (4.2, -0.1) circle (2pt);
          \fill[blue] (4, 0) circle (2pt);
        
          \node at (2, -1) {$i'$};
        \end{tikzpicture}
        }
        \caption{}
        \label{fig:sub2}
    \end{subfigure}
    \hfill
    \begin{subfigure}[b]{0.45\textwidth}
        \centering
        \resizebox{\linewidth}{!}{
        \begin{tikzpicture}[thick]

          \draw (0,0) ellipse (1.2 and 2.5);
          \node[above] at (-0.2, 1.5) {$D_0$};
          
          \draw[white, line width=15pt] (1.0, 0.8) -- (1.5, 0.8);
          \draw[white, line width=15pt] (1.0, -0.4) -- (1.5, -0.4);

          \fill[red] (0, 1.0) circle (2pt);
          \draw[red] (0, 1.0) -- (1.5, 1.0);
          \draw[red] (0.5, 0.6) -- (1.5, 0.6);
          \draw[red] (1.5,0.6) 
            .. controls (4, -1) and (4, 2.6) .. (1.5, 1.0);
          \draw[red] (0.5, 0.6)
            .. controls (-0.2, 0.6) and (-0.2, -0.3) .. (0.5,-0.3);
          \draw[red] (0.5,-0.3) -- (4,-0.3);
          \draw[red] (4.2,-0.1) 
            .. controls (4.8, 3.2) and (7.2, 2.2) .. (7.2, 1) 
            .. controls (7.2, -1) and (5, -0.3)    .. (4,-0.3);

          \fill[blue] (0, -0.5) circle (2pt);
          \draw[blue, dashed] (3,-0.5) -- (5,-0.6);
          \draw[blue] (0,-0.5) -- (3,-0.5);
          \draw[blue] (4,0) 
            .. controls (4.8, 3.5) and (7.5, 2.5) .. (7.5, 1) 
            .. controls (7.5, -1) and (5, -0.6)    .. (5,-0.6);
          \draw[blue] (3,-0.5)
            .. controls (3.5, -2) and (4.5, -2) .. (5,-0.6);

          \draw[white, line width=10pt] (3.5, -1) -- (3.5, 0);
          \draw (4.3, -0.17) arc (45:345:0.5);

          \fill[red] (4.2, -0.1) circle (2pt);
          \fill[blue] (4, 0) circle (2pt);
        
          \node at (2, -1) {$i''$};
          \node[blue] at (4, -2) {$\alpha_1$};
          \node[blue] at (4.6, -0.9) {$\alpha_0$};
        \end{tikzpicture}
        }
        \caption{}
        \label{fig:sub3}
    \end{subfigure}
    \hfill
    \begin{subfigure}[b]{0.45\textwidth}
        \centering
        \resizebox{\linewidth}{!}{
        \begin{tikzpicture}[thick]

          \draw (0,0) ellipse (1.2 and 2.5);
          \node[above] at (-0.2, 1.5) {$D_0$};
        
          \draw[white, line width=15pt] (1.0, -0.4) -- (1.5, -0.4);
          \draw[white, line width=15pt] (1.0, 0.8) -- (1.5, 0.8);

          \fill[red] (0, 1.0) circle (2pt);
          \draw[red] (0, 1.0) -- (1.5, 1.0);
          \draw[red] (0.5, 0.6) -- (1.5, 0.6);
          \draw[red] (1.5,0.6) 
            .. controls (4, -1) and (4, 2.6) .. (1.5, 1.0);
          \draw[red] (0.5, 0.6)
            .. controls (-0.2, 0.6) and (-0.2, -0.3) .. (0.5,-0.3);
          \draw[red] (0.5,-0.3) -- (3.1,-0.3);
          \draw[red] (4.2,-0.1) 
            .. controls (4.8, 3.2) and (7.2, 2.2) .. (7.2, 1) 
            .. controls (7.2, -1) and (5, -0.3)    .. (4.9,-0.3);
          \draw[red] (3.1,-0.3)
            .. controls (3.5, -1.7) and (4.5, -1.7) .. (4.9,-0.3);

          \fill[blue] (0, -0.5) circle (2pt);
          \draw[blue] (0,-0.5) -- (3,-0.5);
          \draw[blue] (4,0) 
            .. controls (4.8, 3.5) and (7.5, 2.5) .. (7.5, 1) 
            .. controls (7.5, -1) and (5, -0.6)    .. (5,-0.6);
          \draw[blue] (3,-0.5)
            .. controls (3.5, -2) and (4.5, -2) .. (5,-0.6);
          
          \draw (4.4, -0.5) arc (0:360:0.5);

          \fill[red] (4.2, -0.1) circle (2pt);
          \fill[blue] (4, 0) circle (2pt);
        
          \node at (2, -1) {$i'''$};
        \end{tikzpicture}
        }
        \caption{}
        \label{fig:sub4}
    \end{subfigure}
    \caption{}
    \label{fig:main}
  \end{figure}

  As shown in Figure \ref{fig:sub1}, 
  the embedding $i_g$ is obtained by attaching a tube (represented by the blue curve) 
  from $D_0$ to the meridian of the tube at a point (represented by the black circle).
  By taking an internal stabilization, we attach a red tube to the embedded surface, 
  resulting in the embedding $i_g^{\gamma}$.
  Here it is reasonable to represent the attached tube by a guiding arc, 
  as this does not change the ambient isotopy class of the embedded surface, 
  according to Remark \ref{rem:intstb}.

  We begin by applying an isotopy that pushes one of the bases of the red local tube along the entire blue self-referential tube.
  As the base travels, it traces out a parallel red tube alongside the blue one.
  The resulting embedding $i'$ is shown in Figure \ref{fig:sub2}.

  We can apply Lemma \ref{lem:lbl} to isotope the blue tube out of the meridian sphere $m$.
  We check the conditions of the light bulb lemma here.
  \begin{itemize}
    \item[(1)] The boundary geometric dual of the original embedding $i$ is also a boundary geometric dual of $i'$. 
    \item[(2)] In Figure \ref{fig:sub3}, for convenience, we abuse the notation $\alpha_0$ and $\alpha_1$ to denote both the core of the tubes and the tubes themselves.
    The disk bounded by the curves $\alpha_0$ and $\alpha_1$ intersects the sphere $m$ transversely at one point $q$.
    The image of $i'$ excluding the tube $\alpha_0$ is connected, because of the existence of the red tube.    
  \end{itemize}
  The resulting embedding $i''$ is shown in Figure \ref{fig:sub3}. 

  By applying the light bulb lemma a second time, 
  we can isotope the red tube out of the sphere $m$.
  The resulting embedding $i'''$ is shown in Figure \ref{fig:sub4}.

  Once the self-referential obstruction is resolved, 
  the embedding $i'''$ is simply obtained from tubing $i$ to itself along the concatenation of the blue arc and the red arc. 
  According to Remark \ref{rem:intstb2}, a final isotopy allows us to retract this additional tube entirely to the tube along $\gamma$.
  Therefore, $i_g^{\gamma}$ and $i^{\gamma}$ are isotopic relative to the boundary.
\end{proof}

Since the isotopy from Theorem \ref{thm1} may be taken to be supported in a small neighbourhood of the self‑referential tube,
we immediately obtain a corollary stating that an internal stabilization can resolve multiple self‑referential tubes.

\begin{cor}\label{cor0}
  Suppose $i: \Sigma \hookrightarrow M$ is a neat embedding with a boundary geometric dual.
  Let $\omega$ be an element of $\ZZ[\pi_1(M)]$, and $\gamma: I \hookrightarrow M$ be an arc embedding with endpoints on $i({\rm int}(\Sigma))$.
  Then the two embeddings $i_\omega^{\gamma}$ and $i^{\gamma}$ are isotopic relative to the boundary.
\end{cor}

\begin{rem}\label{rem:omega}
  The embedding $i_\omega$ is obtained by adding self-referential tubes on $i$.
  For detailed construction we refer the readers to \cite[Definition 1.7]{S21}.
\end{rem}

Considering the case for standrad internal stabilizations, Corollary \ref{cor0} and Remark \ref{rem:mkpt2} directly deduce the following.

\begin{cor}\label{cor1}
  Suppose $i: \Sigma \hookrightarrow M$ is a neat embedding with a boundary geometric dual.
  Let $\omega$ be an element of $\ZZ[\pi_1(M)]$, and $n$ be a positive integer.
  Then the two embeddings $i_\omega^{\bullet n}$ and $i^{\bullet n}$ are isotopic relative to the boundary.
\end{cor}

\section{Isotopy Classification of High-genus Surfaces in $S^2$-bundles over $\Sigma$}\label{sec:main}

Let $\Sigma$ be the closed surface with genus $g \geq 1$.
Up to bundle isomorphism, there are two orientable $S^2$-bundles over $\Sigma$, namely, the trivial one $\Sigma \times S^2$ and the nontrivial one $\Sigma \ltimes S^2$.
These two bundles can be distinguished by the second Stiefel-Whitney class.
Moreover, as smooth $4$-manifolds, $\Sigma \times S^2$ and $\Sigma \ltimes S^2$ are not diffeomorphic,
since they have non-isomorphic intersection forms.
We refer the readers to \cite[Subsection 2.3]{Y26} for more details.
Let $X$ be $\Sigma \times S^2$ or $\Sigma \ltimes S^2$.
Pick a basepoint $b \in \Sigma$.
Denote the $S^2$-fiber at $b$ as $G$.

In this section, we will apply Theorem \ref{thm1} to classify the embedded surfaces of genus $g' > g$ in $X$ with geometric dual $G$.
Let $\Sigma'$ be the closed surface with genus $g'$.

\subsection{Homotopy classification}

Let $\langle \Sigma',M \rangle$ be the set containing the pointed homotopy classes from $\Sigma'$ to $M$.
Here $M$ is an arbitrary $4$-manifold. 
The following proposition gives a specific description of $\langle \Sigma',M \rangle$.

\begin{prop}\label{prop:hmtp}
  There is a bijection $$ \phi: \langle \Sigma',M \rangle \rightarrow {\rm Hom}(\pi_1(\Sigma'),\pi_1(M)) \times \pi_2(M).$$
\end{prop}

The proof of Proposition \ref{prop:hmtp} is essentially given in \cite[Proposition 3.5]{Y26}.
However, the proof of \cite[Proposition 3.5]{Y26} is stated only for special surfaces and $4$-manifolds. 
For the sake of rigor, we sketch the proof of Proposition \ref{prop:hmtp} in the general case.

\begin{proof}[Sketch of the proof.]
  There is a map $\phi_1: \langle \Sigma',M \rangle \rightarrow {\rm Hom}(\pi_1(\Sigma'),\pi_1(M))$,
  sending a pointed map $f:\Sigma' \rightarrow M$ to the induced homomorphism on the fundamental groups. 
  Note that for any $\varphi \in {\rm Hom}(\pi_1(\Sigma'),\pi_1(M))$, we fix a pointed map $f_\varphi : \Sigma' \rightarrow M$ such that $\phi_1([f_\varphi]) = \varphi$.
  
  Given pointed maps $g_0,g_1:\Sigma' \rightarrow M$ satisfying $\phi_1([g_0]) = \phi_1([g_1])$.
  After an original homotopy, we may assume that $g_0$ and $g_1$ are identical on the $1$-skeleton of $\Sigma'$.
  Then the characteristic map $j: D^2 = e^2 \rightarrow \Sigma'$ of the unique $2$-cell of $\Sigma'$
  satisfies $g_0' \circ j |_{\partial D^2} = g_1 \circ j|_{\partial D^2}$, and induces a map $\psi: S^2 \rightarrow M$.
  We denote that $[\psi] = \phi_2([g_0],[g_1])$.

  We can define $\phi([f]) = (\phi_1([f]), \phi_2([f],[f_{\phi_1([f])}]))$.
  One can show that $\phi$ is well-defined, injective and surjective via the same argument as the proof of \cite[Proposition 3.5]{Y26}.
\end{proof}

Here are direct corollaries of Proposition \ref{prop:hmtp}.

\begin{cor}\label{cor:hmpt1}
  For $X = \Sigma \times S^2$, the pointed homotopy class of a smooth map $f: \Sigma' \rightarrow X$ is determined by 
  the induced homomorphism $({\rm pr}_1 \circ f)_* : \pi_1(\Sigma') \rightarrow \pi_1(\Sigma)$
  and the degree of the map ${\rm pr}_2 \circ f: \Sigma' \rightarrow S^2$. 
\end{cor}

For $X = \Sigma \ltimes S^2$, we may fix a model for $\Sigma \ltimes S^2$.
Pick a homotopically nontrivial loop $\phi: S^1 \rightarrow SO(3)$ to be
\begin{equation}
  \phi(\theta) = \begin{pmatrix}
  \cos \theta & \sin \theta & 0 \\
  -\sin \theta & \cos \theta & 0 \\
  0 & 0 & 1
  \end{pmatrix},
\end{equation}
and construct the model for $\Sigma \ltimes S^2$ by defining 
\begin{equation}
  \Sigma \ltimes S^2 = (\Sigma-D)\times S^2 \cup_\phi D\times S^2.
\end{equation}
There is an embedding $\S_0$ given by 
\begin{equation}\label{def:s0}
  \S_0: \Sigma \hookrightarrow \Sigma \ltimes S^2, \quad x \mapsto (x,(0,0,1)^T).
\end{equation}

\begin{cor}\label{cor:hmtp2}
  For $X = \Sigma \ltimes S^2$ and a smooth map $f: \Sigma' \rightarrow X$, its pointed homotopy class is determined by
  the induced homomorphism $({\rm pr}_1 \circ f)_* : \pi_1(\Sigma') \rightarrow \pi_1(\Sigma)$
  and the algebraic intersection number of $f$ and $\S_0$.
\end{cor}

\subsection{Isotopy classification of embedded $\Sigma$ in $X$}

We first review the construction of a family standard embeddings of $\Sigma$ to $X$ in \cite{LWXZ25} and \cite{Y26}.

\begin{defn}[\cite{LWXZ25}]
  Here let $X = \Sigma \times S^2$. 
  We define the embedding $\S_0:\Sigma \hookrightarrow X$ by $\S_0(x) = (x,{\rm pt})$ for a point ${\rm pt} \in S^2$.
  Given a positive integer $d$, choose $d$ distinct points $p_1,\dots,p_d$ in $\Sigma$, 
  none of which equals $b = \operatorname{pr}(G)$.
  Let $\S_d: \Sigma \rightarrow \Sigma \ltimes S^2$ to be the embedding obtained by 
  resolving the intersections of $\S_0$ and the $S^2$-fibers of $\Sigma \ltimes S^2$ over $p_1,\cdots,p_d$. 
  Given a negative integer $d$, we define $\S_d$ in the same way, 
  but resolve the intersections of $\S_0$ and the $S^2$ fibers over $p_1,\cdots,p_{-d}$ with the opposite orientation.
\end{defn}

\begin{defn}[\cite{Y26}]
  Here let $X = \Sigma \ltimes S^2$. 
  Recall that the embedding $\S_0:\Sigma \hookrightarrow X$ is defined in (\ref{def:s0}).
  Given an integer $d$, we choose $|d|$ distinct points $p_1,\dots, p_{|d|}$ in $\Sigma \setminus \{b\}$.
  After resolving the intersections of $\S_0$ with fibers (may with the opposite orientation, up to the sign of $d$) over these points, 
  we obtain the embedding $\S_d$.
\end{defn}

Then we can review the following results according to \cite[Theorem 6.1]{LWXZ25} and \cite[Theorem 4.1]{Y26}.

\begin{prop}[\cite{LWXZ25, Y26}]\label{prop:2526}
  Let $i:\Sigma \hookrightarrow X$ be an embedding satisfying the following conditions.
  \begin{itemize}
    \item[(1)] The $2$-sphere $G$ is a geometric dual of $i$.
    \item[(2)] The embedding $i$ agrees with $\S_0$ in a neighbourhood of $b$ in $\Sigma$.
    \item[(3)] The embedding $i$ is homotopic to $\S_d$ relative to $b$ for an integer $d$.  
  \end{itemize}
  Then after an ambient isotopy, we obtain an embedding $j$ satisfying the following.
  \begin{itemize}
    \item[(a)] The resulting embedding $j$ agree with $\S_d$ in a neighbourhood of ${\rm sk}^1$. 
    Here ${\rm sk}^1$ is the $1$-skeleton of the standard cellular complex structure of $\Sigma$ with the only $0$-cell $b$.
    \item[(b)] We denote the resulting $4$-manifold after excluding a tubular neighbourhood of $\S_d({\rm sk}^1)$ by $X''$. 
    By restricting $j$ and $\S_d$ to the pre-image of $X''$, 
    we obtain two neat embeddings $j'$ and $\S_d'$ of disk,
    which are identical on the boundary.
    There is an element $\omega$ in $\ZZ[\pi_1(X'')] \cong \ZZ[\pi_1(X)]$, such that $j' = (\S_d')_\omega$.
    The notation $(\S_d')_\omega$ is mentioned in Remark \ref{rem:omega}.
  \end{itemize}
\end{prop}

\begin{rem}
  We remark here that Proposition \ref{prop:2526} is a direct corollary of the proof of \cite[Theorem 6.1]{LWXZ25} and \cite[Theorem 4.1]{Y26}.
  We reformulate the statement of the conclusion here so as to avoid introducing the Dax invariant, which is unnecessary for the subsequent content. 
  Since the fundamental group of $X''$ contains no $2$-torsion, 
  the result $(b)$ of Proposition \ref{prop:2526} also can be deduced by \cite[Theorem 0.2]{S21} and the result $(a)$.
\end{rem}

\subsection{An initial isotopy on the $1$-skeleton of $\Sigma'$}
In this subsection, we study the homotopical properties of embeddings $\Sigma' \hookrightarrow X$ with a geometric dual.
It deduces that we may isotope the $1$-skeleton of such embedded surfaces to a standard one. 
Recall that $X$ refers to $\Sigma \times S^2$ or $\Sigma \ltimes S^2$, and $G$ refers to an $S^2$-fiber.
Denote the bundle projection map by ${\rm pr}: X \rightarrow \Sigma$.

\begin{lem}\label{lem:deg1}
  Let $f: \Sigma' \hookrightarrow X$ be an embedding with the geometric dual $G$.
  Then the composition map ${\rm pr} \circ f: \Sigma' \rightarrow \Sigma$ is of degree $1$.
\end{lem}

\begin{proof}
Recall that $b={\rm pr}(G)\in \Sigma$, so that $G={\rm pr}^{-1}(b)$.  Since
$f(\Sigma')$ has geometric dual $G$, $b$ is a regular value of
${\rm pr}\circ f$ and$({\rm pr}\circ f)^{-1}(b)=f^{-1}(G)$
consists of a single point.
At the unique point of $({\rm pr}\circ f)^{-1}(b)$, the local degree of
${\rm pr}\circ f$ is equal to the local intersection sign of
$f(\Sigma')$ with the fiber $G$.  
Therefore, $\deg({\rm pr}\circ f) = 1$.
\end{proof}

We recall the definition of the pinch map and the characterization of degree‑one maps between surfaces, 
following the work \cite{E79} of Edmonds.
Let $S,S'$ be compact, connected and oriented surfaces with one boundary component.

\begin{defn}[\cite{E79}]
A map $p\colon S'\to S$ is called a \emph{pinch map}
if there exists a compact connected surface $S_1\subset {\rm int}(S')$ whose boundary
is a single simple closed curve in $S'$, such that
$S = S'/S_1$ and $p$ is the quotient map which collapses $S_1$ to one point.
\end{defn}

\begin{thm}[\cite{E79}]\label{thm:ed}
Let $f\colon S'\to S$ be a map between surfaces.  
Suppose $\deg(f)=1$, $f^{-1}(\partial S) = \partial S'$ and $f|_{\partial S'}$ is a diffeomorphism.
Then $f$ is homotopic to a pinch map $p\colon S'\to S$ relative to $\partial S'$.
\end{thm}

The following proposition shows that, up to a reparameterization and an isotopy,
we may tame any embedding $f:\Sigma' \hookrightarrow X$ to a standard case after taking internal stabilizations.
Recall that $g,g'$ are genera of $\Sigma$ and $\Sigma'$ respectively.

\begin{lem}\label{lem:first isotopy}
  Let $f: \Sigma' \hookrightarrow X$ be an embedding with the geometric dual $G$.
  There is a diffeomorphism $r: \Sigma' \rightarrow \Sigma'$,
  such that $f \circ r: \Sigma' \hookrightarrow$ is isotopic to $\tilde{f}^{\bullet (g'-g)}:\Sigma' \hookrightarrow X$.
  Here $\tilde{f}:\Sigma \hookrightarrow X$ is an embedding,
  still with geometric dual $G$, 
  such that ${\rm pr}\circ \tilde{f}$ induces identity map between fundamental groups. 
  Moreover, 
  the reparameterization $r$ can be supported away from a neighbourhood of $f^{-1}(G)$,
  and the isotopy can be supported away from a neighbourhood of $G$.
\end{lem}

\begin{proof}

Let $q\in\Sigma'$ be the unique point such that $f(q)\in G$, and recall that $b={\rm pr}(G)$.  
Since $f(\Sigma')$ intersects $G$ transversely and positively at $f(q)$, 
we may take a sufficiently small closed disk $D \subset \Sigma$ centered at $b$, 
satisfying that $({\rm pr} \circ f)^{-1}(D)=D'$.
Here $D'\subset\Sigma'$ is a disk containing $q$, and
$$({\rm pr} \circ f)|_{D'}\colon D'\longrightarrow D$$
is an orientation-preserving diffeomorphism.  
We write
$$N=\nu(G)={\rm pr}^{-1}(D),\qquad
X'=X\setminus\operatorname{int}(N),$$
and
$$S'=\Sigma'\setminus\operatorname{int}(D'),\qquad
S=\Sigma\setminus\operatorname{int}(D).$$
Then $f|_{S'}$ is a neat embedding into $X'$.

According to Lemma \ref{lem:deg1}, the map ${\rm pr}\circ f\colon \Sigma'\rightarrow \Sigma$ has degree $1$. 
It deduces that
$$({\rm pr}\circ f)|_{S'}\colon (S',\partial S')\longrightarrow (S,\partial S)$$
is a map satisfying the conditions of Theorem \ref{thm:ed}.
Hence, $({\rm pr}\circ f)|_{S'}$ is homotopic to a pinch map $ p\colon S'\rightarrow S$ relative to $\partial S'$.
Let $C\subset{\rm int}(S')$ be the compact connected subsurface collapsed by $p$.
Since $S'/C = S$ and $\partial C$ is connected, $C$ has genus $g'-g$.

Fix a standard $1$-skeleton
$${\rm sk}^1 =\vee_{i=1}^{g'}(a_i\vee b_i)\subset\Sigma',$$
based at a point of $\partial S'$,
and adapted to the decomposition of $\Sigma'$ into a genus-$g$ subsurface $\Sigma' \setminus C$ containing $D'$ and the subsurface $C$.  
We may identify $\pi_1(S')$ to the presentation with free generators $a_1,b_1,\cdots,a_{g'},b_{g'}$.
According to the Dehn-Nielson theorem, 
there is a diffeomorphism $r\colon\Sigma'\to\Sigma'$,
whose restriction to $D'$ is the identity map.
After precomposing with $r$,
we may assume that
$$(p \circ r)_*(a_i)=a_i,\qquad
(p \circ r)_*(b_i)=b_i\quad (1\leq i\leq g),$$
whereas
$$({\rm pr}\circ f \circ r)_*(a_i)=({\rm pr}\circ f \circ r)_*(b_i)=1
\quad (g<i\leq g').$$
Here, on the right-hand side, we use the standard free generators of $\pi_1(S)$.

We now construct a section $\S\colon\Sigma\hookrightarrow X$ of the $S^2$-bundle $X$.
We may identify $D'$ to $D$ via the diffeomorphism $({\rm pr}\circ f)|_{D'}$. 
Under such identification, the disk $f(D')$ is a section of the bundle $N\rightarrow D$.  
This local section extends to a global section $\S\colon\Sigma\hookrightarrow X$ for the $S^2$-bundle $X$ over $\Sigma$.
Indeed, the restriction of the $S^2$-bundle to $S$ is trivial, 
and the boundary value of the local section extends over $S$ since $\pi_1(S^2)=0$. 
After taking $g'-g$ internal stabilizations at the points away from the disk $f(D')$, we obtain the embedding $\S^{\bullet g'-g}$.

Consequently, with respect to the above standard $1$-skeleton, we have
$$({\rm pr}\circ\S^{\bullet g'-g})_*(a_i)=a_i,\qquad
({\rm pr}\circ\S^{\bullet g'-g})_*(b_i)=b_i
\quad (1\leq i\leq g),$$
and
$$
({\rm pr}\circ\S^{\bullet g'-g})_*(a_i)
=({\rm pr}\circ\S^{\bullet g'-g})_*(b_i)=1
\quad (g<i\leq g'),$$
and $\S^{\bullet g'-g}|_{D'}$ is identical to $f|_{D'}$.

Since ${\rm pr}\circ f \simeq p$, it follows that
$$({\rm pr}\circ f\circ r|_{{\rm sk}^1})_* = ({\rm pr}\circ\S^{\bullet g'-g}|_{{\rm sk}^1})_*:\pi_1({\rm sk}^1) \longrightarrow \pi_1(S).$$
Because the map ${\rm pr}_*\colon\pi_1(X')\to\pi_1(S)$ is an isomorphism, we have
$$(f\circ r|_{{\rm sk}^1})_* = (\S^{\bullet g'-g}|_{{\rm sk}^1})_*:\pi_1({\rm sk}^1) \longrightarrow \pi_1(X').$$
Thus the two embeddings of ${\rm sk}^1$ into $X'$ are homotopic. 
By a general-position argument, this homotopy can be perturbed to an isotopy supported away from $\partial X'$. 
After applying the isotopy extension theorem, we obtain an ambient isotopy $H: X \times I \rightarrow X$, 
such that $H$ fixes $N$, $H(-,0) = {\rm id}_X$, and $H(f \circ r(-),1)|_{{\rm sk}^1} = \S^{\bullet g'-g}|_{{\rm sk}^1}$.

Denote the embedding $H(f \circ r(-),1): \Sigma' \hookrightarrow X$ by $\phi$. 
Since the isotopy $H$ fixes $N$, $G$ is still a geometric dual of $\phi$.
Regarding the induced map on fundamental groups,
$\phi_*$ maps $a_i,b_i$ to the unit for $g < i \leq g'$.
The embedding $\phi$ together with the simple closed curves $\phi(a_{g+1}), \cdots, \phi(a_{g'})$ satisfy the conditions of Lemma \ref{lem:gabai}.
We can find $g'-g$ disjoint compressing disks $D_{g+1}, \cdots, D_{g'}$ in $X'$ 
bounded by $\phi(a_{g+1}), \cdots, \phi(a_{g'})$ respectively.
According to Proposition \ref{prop2.9}, $\phi$ is isotpic to $(\phi_{D_{g+1},\cdots,D_{g'}})^{\bullet g'-g}$,
via an isotopy supported away from $N$.
Denoting $\phi_{D_{g+1},\cdots,D_{g'}}$ by $\tilde{f}$, we conclude our proof.
\end{proof}

\subsection{Isotopy classification of embedded $\Sigma'$ in $X$}

In this subsection, we state in detail our main result for isotopy classification of embedded $\Sigma'$ in $X$ with geometric dual $G$.

\begin{thm}\label{thm:main1}
  Let $\Sigma'_1, \Sigma'_2 \subset X$ be two embedded closed surfaces with genus-$g'$ in $X$.
  Suppose the following conditions holds.
  \begin{itemize}
    \item[(1)] The sphere $G$ is a common geometric dual of $\Sigma'_1$ and $\Sigma'_2$.
    \item[(2)] For $i=1,2$, there is a parameterization $f_i: \Sigma \hookrightarrow X$ of $\Sigma'_i$,
    such that $f_1$ and $f_2$ are homotopic.
  \end{itemize}
  Then $\Sigma'_1$ and $\Sigma'_2$ are ambiently isotopic.
\end{thm}

\begin{proof}
  We first deal with the intersections of the embedded surfaces with a tubular neighbourhood of the geometric dual $G$. 
  Denote $\Sigma'_i \cap G$ by $q_i$ for $i = 1,2$.
  There is an isotopy $h: G \times I \rightarrow G$ sending $q_1$ to $q_2$.
  By applying the isotopy extension theorem and a reparameterization, 
  we may assume that $f_1(b') = f_2(b') \in G$.
  Here $b'$ is a fixed basepoint in $\Sigma'$. 
  Since $f_1$ and $f_2$ intersect $G$ transversely at the same point,
  up to an ambient isotopy supported in a tubular neighbourhood $N$ of $G$,
  we may assume that there is a $2$-disk $D'$ in $\Sigma'$ containing $b'$, 
  such that $f_1^{-1}(N) = f_2^{-1}(N) = D'$, 
  and $f_1|_{D'}, f_2|_{D'} : D' \hookrightarrow N$ are neat embeddings with the same image.
  Because the mapping class group of the oriented $2$-disk is trivial,
  we may isotope the parameterization $f_1$, 
  such that the resulting parameterization agrees with $f_2$ on $D'$.
  Thus, we can assume that $f_1, f_2$ are identical on $D'$. 

  Apply Lemma \ref{lem:first isotopy} to the embeddings $f_1$ and $f_2$. 
  For $i=1,2$, up to reparameterization, $f_i$ is isotopic to $\tilde{f_i}^{\bullet g'-g}$.
  Since the reparameterization and the isotopy can be supported away from a neighbourhood of $G$,
  the embeddings $\tilde{f_1}^{\bullet g'-g}$ and $\tilde{f_2}^{\bullet g'-g}$ are still identical on $D'$.
  Here $\tilde{f_i}$ is an embedding of $\Sigma$ such that ${\rm pr} \circ \tilde{f_i}$ induces identity map on fundamental groups.
  Because internal stabilizations can be taken at the points away from a disk in $\Sigma$, 
  we may assume that $\tilde{f_1}$ and $\tilde{f_2}$ are identical on a disk $D \subset \Sigma$.

  In order to apply Proposition \ref{prop:2526} to $\tilde{f_1}$ and $\tilde{f_2}$, we now only need to check condition $(3)$ of Proposition \ref{prop:2526}. 
  Consider the elements in $\langle \Sigma',X \rangle$ they represent.
  For $X = \Sigma \times S^2$, we apply Corollary \ref{cor:hmpt1}.
  Since isotopy, orientation-preserving reparameterization and taking internal stabilizations do not change the degree of the map, 
  the degrees of ${\rm pr} \circ \tilde{f_1}$ and ${\rm pr} \circ \tilde{f_2}$ are identical.
  Therefore, $\tilde{f_1}, \tilde{f_2}$ represent the same element in $\langle \Sigma',X \rangle$.
  For $X = \Sigma \ltimes S^2$, the same conclusion holds, 
  because isotopy, orientation-preserving reparameterization and taking internal stabilizations do not change the algebraic intersection number. 
  Then we apply Proposition \ref{prop:2526}, 
  ambiently isotope $\tilde{f_1}$ and $\tilde{f_2}$ to embeddings identical to $\S_d$ on a neighbourhood of $1$-skeleton, 
  for an integer $d$.
  We still denote the resulting embeddings by $\tilde{f_1}$ and $\tilde{f_2}$.

  After excluding a tubular neighbourhood $\nu({\S_d}({\rm sk}^1))$ in $X$, 
  we obtain neat embeddings $D_1, D_2: D^2 \hookrightarrow X''$ homotopic relative to the boundary.
  Note that internal stabilizations can be taken at the points away from the neighbourhood of $1$-skeleton of $\Sigma$.
  The embeddings $\tilde{f_1}^{\bullet g'-g}$ and $\tilde{f_2}^{\bullet g'-g}$ give $D_1^{\bullet g'-g}$ and $D_2^{\bullet g'-g}$.
  According to Proposition \ref{prop:2526}, there is an element $\omega$ in $\ZZ[\pi_1(X'')]$, 
  such that $D_1$ is isotopic to $(D_2)_\omega$ relative to the boundary.
  Applying Theorem \ref{thm1} and Corollary \ref{cor1}, we deduce that $D_1^{\bullet g'-g}$ is isotopic to $(D_1)_\omega^{\bullet g'-g}$ relative to the boundary.
  By Remark \ref{rem:mkpt2}, $D_1^{\bullet g'-g}$ and $D_2^{\bullet g'-g}$ are isotopic relative to the boundary in $X'' \subset X$.
  It concludes our proof. 
\end{proof}

Recall that $\mathscr{E}$ is the set consisting of the embedded genus-$g'$ surfaces in $X$ with the geometric dual $G$.
We have the following direct corollary.

\begin{cor}
  There is a bijection 
  $$ \Phi: \mathscr{E}/\textit{ambient isotopy} \longrightarrow \pi_2(X) \cong \ZZ. $$
\end{cor}

\begin{proof}
  Theorem \ref{thm:main1} shows that homotopy implies ambient isotopy for elements in $\mathscr{E}$.
  Therefore, we only need to clarify the homotopy classification for unparameterized embedded surfaces in $\mathscr{E}$.
  According to Lemma \ref{lem:first isotopy}, 
  the factors in ${\rm Hom}(\pi_1(\Sigma'),\pi_1(X)) = {\rm Hom}(\pi_1(\Sigma'),\pi_1(\Sigma))$ in Proposition \ref{prop:hmtp} 
  can always be taken to be the induced homomorphism of a standard pinch map.
  Thus, the map $\Phi$ is given by ${\rm pr}_2 \circ \phi$.
  Here $\phi$ is the bijection in Proposition \ref{prop:hmtp} and ${\rm pr}_2$ is the projection from the codomain of $\phi$ to its second factor.
\end{proof}

\newpage

\bibliographystyle{plain}
\bibliography{daxinv}

\end{document}